\documentclass{amsart}

\usepackage[utf8]{inputenc}
\usepackage[OT1]{fontenc}

\usepackage[hidelinks]{hyperref}
\usepackage{enumitem}
\setlist[enumerate]{label={\rm(\alph*)}}
\usepackage[all]{xy}
\usepackage{amsmath}
\usepackage{amssymb}

\DeclareFontFamily{OT1}{pzc}{}
\DeclareFontShape{OT1}{pzc}{m}{it}%
             {<-> s * [1,150] pzcmi7t}{}
\DeclareMathAlphabet{\mathpzc}{OT1}{pzc}%
                                 {m}{it}
\theoremstyle{plain}
\newtheorem{thm}{Theorem}[section]
\newtheorem{prop}[thm]{Proposition}
\newtheorem{lemma}[thm]{Lemma}
\newtheorem{coro}[thm]{Corollary}

\theoremstyle{definition}
\newtheorem{defi}[thm]{Definition}
\newtheorem{paragr}[thm]{}
\theoremstyle{remark}
\newtheorem{rem}[thm]{Remark}
\newtheorem{ex}[thm]{Example}

\let\ndef\emph
\let\hookto\hookrightarrow
\let\nbd\nobreakdash

\newcommand\OmegaP{\Omega_p}
\newcommand\OmegaPCl{\overline{\Omega_p}}
\newcommand\OmegaCl{\overline{\Omega}}
\DeclareMathOperator\cl{cl}

\newcommand\oo{$\infty$\nbd}

\newcommand\id[1]{1_{#1}}
\newcommand\eps\varepsilon
\DeclareMathOperator{\Hom}{\mathsf{Hom}}

\newcommand\W{\mathsf{W}}
\newcommand\Cof{\mathsf{Cof}}
\newcommand\Fib{\mathsf{Fib}}

\newcommand\Cat{{\mathpzc{Cat}}}
\newcommand\Oper{{\mathpzc{Oper}}}

\newcommand\pref[1]{\widehat{#1}}
\newcommand\sSet{\pref{\Delta}}
\newcommand\dSet{\pref{\Omega}}
\newcommand\M{\mathcal{M}}

\newcommand\Sh[2]{\mathrm{Sh}_{#1, #2}}

\newcommand\zbox[1]{\makebox[0pt][l]{#1}}
\newcommand\pbox[1]{\zbox{#1}}
\newcommand\annot[1]{\zbox{$\,\scriptstyle #1$}}

\newcommand\limind\varinjlim

\newcommand\operadic{\mathrm{oper}}

\newcommand\SC[1]{\mathrm{Sc}({#1})}

\newcommand\tensBV{\otimes_{\mathrm{BV}}}

\title{The dendroidal category is a test category}

\author[D. Ara]{Dimitri Ara}
\address{Aix Marseille Univ, CNRS, Centrale Marseille, I2M, Marseille,
France}
\email{dimitri.ara@univ-amu.fr}

\author[D.-C. Cisinski]{Denis-Charles Cisinski}
\address{Fakultät für Mathematik, Universität Regensburg, 93040 Regensburg,
Deutschland}
\email{denis-charles.cisinski@mathematik.uni-regensburg.de}

\author[I. Moerdijk]{Ieke Moerdijk}
\address{Department of Mathematics, Utrecht University, PO BOX 80.010, 3508 TA
Utrecht, The Netherlands}
\email{i.moerdijk@uu.nl}

\begin{document}

\begin{abstract}
We prove that the category of trees $\Omega$ is a test category in the sense
of Grothendieck. This implies that the category of dendroidal sets is
endowed with the structure of a model category Quillen-equivalent to spaces.
We show that this model category structure, up to a change of cofibrations,
can be obtained as an explicit left Bousfield localisation of the operadic
model category structure.
\end{abstract}

\maketitle

\section*{Introduction}

The notion of a test category was introduced by Grothendieck in his
influential manuscript \cite{GrothPS}, with the aim of axiomatising those small
categories which could play a role similar to the category $\Delta$ of
simplices and serve as building blocks to describe all homotopy types of
spaces. The theory of test categories has been described and further
developed in \cite{Maltsi}, \cite{Cisinski},
\cite{JardineTest}, \cite{MaltsiCube}, \cite{CisMaltsiTheta}
and \cite{AraThtld}. Here one can find a list of examples of test
categories, which includes, in addition to $\Delta$, familiar categories
such as the category of cubes parametrising cubical sets, and Joyal's
category $\Theta_n$ parametrising a notion of $n$-dimensional category. The
goal of this paper is two-fold.  First of all, we wish to add some further
examples to this list, by showing that the category $\Omega$ of trees which
parametrises dendroidal sets is a test category. The argument will also show
that variations of $\Omega$ such as similar categories of planar trees and
closed trees are test categories.  It follows that there is a Quillen model
structure on the category of dendroidal sets which models the homotopy
category of spaces.  The second goal of this paper is to explain the
relation of this model structure to the model structure that the category of
dendroidal sets was originally designed for, namely the so-called operadic
model structure which models the homotopy theory of topological (or
simplicial) coloured operads.  Indeed, we will show that up to a small
change in the class of cofibrations only, the first model structure coming
from the fact that $\Omega$ is a test category can be obtained as a left
Bousfield localisation of the second, operadic model structure.

The plan of our paper is as follows. In the first section, we will review
the basic definitions of the theory of test categories. In Section 2, we
will present a proof of the fact that the simplex category is a test
category which is somewhat different from the ones occurring in the
literature, and is based on the fact that the product of two simplices can
be written as a union of other simplices indexed by shuffles.  The proof
that $\Omega$ is a test category will be broken up into two parts. The first
part shows the fact that the classifying space of $\Omega$ is contractible.
This fact has been known for quite some time, but a proof has never been
published. The second part of the proof will again use shuffles and follows
the same pattern as the argument for simplices. In the final section, we
discuss the relation to the operadic model structure on dendroidal sets
mentioned above.

The results of this paper go back quite a while, and were presented in 2013
at the conference celebrating the 65th birthday of G.~Maltsiniotis.  We
would like to thank G.~Maltsiniotis for encouraging us to write up the
results, and apologise for the fact that it has taken us a while.

We are grateful to the referee for carefully reading the paper.

\section{Preliminaries on test categories}

In this section, we review some basics of the theory of test categories
introduced by Grothendieck in \cite{GrothPS}. For more detailed expositions
and proofs, we refer the reader to~\cite{Maltsi} and \cite{Cisinski}.

\medbreak

We begin with a bit of notation and terminology.

\begin{paragr}
We denote by $N : \Cat \to \pref{\Delta}$ the nerve functor from small
categories to simplicial sets. A functor $u : A \to B$ between small
categories is said to be a \ndef{weak equivalence} if its nerve $N(u)$ is a
weak homotopy equivalence of simplicial sets. We say that a small category
$A$ is \ndef{aspherical} (or \ndef{weakly contractible}) if the unique
functor from $A$ to the terminal category is a weak equivalence.
\end{paragr}

\begin{paragr}
Let $u : A \to B$ be a functor. If $b$ is an object
of $B$, we denote by $A/b$ the category $A \times_B (B/b)$, where $B/b$ is
the category of objects over $b$. This category is sometimes denoted by $u
\downarrow b$. If $F$ is a presheaf on $A$ and $u$ is the Yoneda embedding,
the category $A/F$ is the category of elements of $F$.
\end{paragr}

We now introduce the basic definitions of the theory of test categories.

\begin{paragr}\label{paragr:def_weq_psh}
Let $A$ be a small category. We have a pair of adjoint
functors
\[
\begin{matrix}
 i^{}_A : & \pref{A}  & \to & \Cat, & \qquad &
i^*_A : & \Cat &  \to & \pref{A}\hfill\\
 & F &  \mapsto &  A/F & \qquad &
 & C & \mapsto & \big(a \mapsto \Hom_\Cat(A/a, C)\big)
\end{matrix}
\]
between presheaves on $A$ and small categories. A morphism $f : X \to Y$ of
presheaves on $A$ is said to be a \ndef{weak equivalence} if $i^{}_A(f) :
A/X \to A/Y$ is a weak equivalence of categories. A presheaf $X$ on $A$ is
\ndef{aspherical} if the category~$A/X$ is aspherical.
\end{paragr}

\begin{defi}
Let $A$ be a small category.
\begin{enumerate}
  \item The category $A$ is said to be a \ndef{weak test category} if, for
  every small category~$C$, the counit functor $\eps^{}_C : i^{}_A i^\ast_A
  C \to C$ is a weak equivalence.
  \item The category $A$ is said to be a \ndef{local test category} if, for
    every object $a$ of~$A$, the slice category $A/a$ is a weak test category.
  \item The category $A$ is said to be a \ndef{test category} if it is both
    a weak test category and a local test category.
\end{enumerate}
\end{defi}

The following proposition shows that to understand test categories, it is
enough to understand local test categories:

\begin{prop}[Grothendieck]\label{prop:test_cat_asph}
A small category $A$ is a test category if and only if the following two
conditions hold:
\begin{enumerate}
  \item $A$ is aspherical;
  \item $A$ is a local test category.
\end{enumerate}
\end{prop}

\begin{proof}
See \cite[Remark 1.5.4]{Maltsi}.
\end{proof}

We now move on to a characterisation of local test categories in terms of
intervals.

\begin{paragr}
An \ndef{interval} of a presheaf category $\pref{A}$ consists of a presheaf
$I$ endowed with two global sections $\partial_0, \partial_1 : \ast \to I$.
The interval is said to be \ndef{separating} if the induced map $\partial_0
\amalg \partial_1 : \ast \amalg \ast \to I$ is a monomorphism.

For instance, the subobject classifier $L_A$ of the topos $\pref{A}$, also
called the \ndef{Lawvere object} of $\pref{A}$, is canonically endowed
with the structure of a separating interval, $\partial_0$ and $\partial_1$
corresponding respectively to the empty subobject and the maximal subobject
of $\ast$.

We say that an interval $I$ is \ndef{locally aspherical} if, for every
presheaf $X$ on~$A$, the projection map $X \times I \to X$ is a weak equivalence. A
direct application of Quillen's Theorem A shows that it is enough to require
this property when $X$ is representable.
\end{paragr}

\begin{thm}[Grothendieck]\label{thm:loc_test}
Let $A$ be a small category. The following conditions are equivalent:
\begin{enumerate}
  \item $A$ is a local test category;
  \item $L_A$ is locally aspherical;
  \item there exists a locally aspherical separating interval in $\pref{A}$.
\end{enumerate}
\end{thm}

\begin{proof}
See \cite[Theorem 1.5.6]{Maltsi}.
\end{proof}

We moreover have from \cite{Cisinski} the following characterisation of
local test categories in terms of model categories:

\begin{thm}\label{thm:model_cat_local_test}
Let $A$ be a small category. The following conditions are equivalent:
\begin{enumerate}
  \item\label{item:tl_def} $A$ is a local test category;
  \item\label{item:tl_mcf} there exists a model category structure on $\pref{A}$ whose
    cofibrations are the monomorphisms and whose weak equivalences are the
    weak equivalences of presheaves as defined in~\ref{paragr:def_weq_psh}.
\end{enumerate}
Moreover, if these conditions are fulfilled, the model category structure of
the second condition is combinatorial and proper.
\end{thm}

\begin{proof}
  See \cite[Corollary 4.2.18]{Cisinski} for the implication
  $\ref{item:tl_def} \Rightarrow \ref{item:tl_mcf}$ (and the fact that the
  model category structure is combinatorial) and
  \cite[Theorem~4.1.19]{Cisinski} for the reciprocal. The properness under
  these assumptions follows from \cite[Corollary~4.2.19 and Example
  4.3.22]{Cisinski}.
\end{proof}

We now characterise test categories in terms of model categories.

\begin{paragr}\label{paragr:def_lambda}
Let $A$ be a small category. Denote by 
\[ \lambda_! = Ni_A : \pref{A} \to \pref{\Delta} \]
the composition of $i_A$ with the nerve functor. The functor $\lambda_!$
preserves colimits (see for instance \cite[Corollary 3.2.10]{Cisinski}) and
hence admits a right adjoint $\lambda^\ast$. We thus have an adjoint pair
\[ \lambda_! : \pref{A} \rightleftarrows \pref{\Delta} : \lambda^\ast. \]
The functor $\lambda_!$ also preserves pullbacks and hence monomorphisms.
\end{paragr}

\begin{thm}\label{thm:test_model}
Let $A$ be a small category. The following conditions are equivalent:
\begin{enumerate}
  \item\label{item:tc} $A$ is a test category;
  \item\label{item:tc_mcf} there exists a model category structure on $\pref{A}$ whose
    cofibrations are the monomorphisms and for which the adjoint pair
    \[ 
    \lambda_! : \pref{A} \rightleftarrows \pref{\Delta} : \lambda^*
    \]
    is a Quillen equivalence, where the category $\pref{\Delta}$ of
    simplicial sets is endowed with the Kan--Quillen model category
    structure.
\end{enumerate}
\end{thm}

\begin{proof}
  The implication $\ref{item:tc} \Rightarrow \ref{item:tc_mcf}$
  is a consequence of \cite[Proposition 4.2.26 and Remark~4.2.27]{Cisinski}.
  Let us prove the converse. Since every object of $\pref{A}$ is cofibrant,
  the left Quillen functor $\lambda_!$ preserves and detects weak
  equivalences. This shows that the weak equivalences of the model category
  structure on $\pref{A}$ are the weak equivalences of presheaves as defined
  in~\ref{paragr:def_weq_psh}. In particular, by
  Theorem~\ref{thm:model_cat_local_test}, $A$ is a local test category.
  By Proposition~\ref{prop:test_cat_asph}, to conclude the proof, it
  suffices to show that $A$ is aspherical. But since $\Delta_0$ is a fibrant
  simplicial set, the morphism $\lambda_! \lambda^\ast(\Delta_0) \to
  \Delta_0$ is a weak equivalence, and since $\lambda_!
  \lambda^\ast(\Delta_0) \simeq N(A)$ we get the result.
\end{proof}

\begin{rem}
We saw in the proof that the model category structure of \ref{item:tc_mcf}
is actually unique (and in particular coincides with the one of
Theorem~\ref{thm:model_cat_local_test}).
\end{rem}

\section{A proof that $\Delta$ is a test category}\label{sec:Delta}

In this section, we give a new proof of the fact that the simplex category
$\Delta$ is a test category. The proof for $\Omega$ will follow a similar
pattern.

\begin{paragr}\label{paragr:Delta_test}
To prove that $\Delta$ is a test category, it suffices to show that for any
$n \ge 0$, the simplicial set $\Delta_1 \times \Delta_n$ is aspherical, where
$\Delta_m$ denotes the standard $m$-simplex. Indeed, since $\Delta$ has a
terminal object, it is aspherical. Moreover, the object~$\Delta_1$ is
clearly a separating interval and our claim thus follows from
Proposition~\ref{prop:test_cat_asph} and Theorem~\ref{thm:loc_test}.
\end{paragr}

From now on, we fix $m, n \ge 0$. We will prove more generally that
$\Delta_m \times \Delta_n$ is aspherical. This fact is well-known and
follows from classical results but we will give an elementary proof in the
spirit of the theory of test categories.

\begin{paragr}\label{paragr:simp_shuffles}
Recall that the simplicial set $\Delta_m \times \Delta_n$ can be written as
a union of subsimplicial sets
\[
  \Delta_m \times \Delta_n = \bigcup_{\sigma \in \Sh{m}{n}} F_\sigma,
\]
where $\Sh{m}{n}$ is the set of $(m, n)$-shuffles and each $F_\sigma$ is
isomorphic to $\Delta_{m+n}$ (see for instance \cite[Chapter II, Section
5]{GabZis}). The only additional fact we will need about these $F_\sigma$'s
is that they all contain the vertex $(0, 0)$ of $\Delta_m \times \Delta_n$.
\end{paragr}

To prove that $\Delta_m \times \Delta_n$ is aspherical, we will use the
following general lemma:

\begin{lemma}\label{lemma:MV}
Let $A$ be a small category and let $F$ be a presheaf on $A$. Suppose that
$F$ can be written as a non-empty finite union of subpresheaves
\[
F = \bigcup_{i \in I} F_i
\]
satisfying the following condition: for every non-empty $J \subseteq I$, the
intersection presheaf
\[
F_J = \bigcap_{j \in J} F_j
\]
is aspherical. Then $F$ is aspherical.
\end{lemma}

\begin{proof}
The case of a binary union is \cite[Proposition 1.2.7]{Maltsi}. The finite
case follows by induction.
\end{proof}

\begin{prop}\label{prop:Delta_tot_asph}
The simplicial set $\Delta_m \times \Delta_n$ is aspherical.
\end{prop}

\begin{proof}
Let $J$ be a non-empty set of $(m, n)$-shuffles. By Lemma~\ref{lemma:MV}, it
suffices to show that
\[ F_J = \bigcap_{j \in J} F_j \subset \Delta_m \times \Delta_n \]
is aspherical. We will show that $F_J$ is actually a representable presheaf.
Indeed, each of the $F_j$'s is the nerve of a subposet of the poset associated
to $\Delta_m \times \Delta_n$ and the intersection $F_J$ is the nerve of the
intersection of these subposets. The result follows from the fact that this
intersection subposet is a non-empty (as it contains~$(0, 0)$) finite linear
order.
\end{proof}

\begin{coro}
The simplex category $\Delta$ is a test category.
\end{coro}

\begin{proof}
This follows from the previous proposition (see \ref{paragr:Delta_test}).
\end{proof}

\begin{rem}
  The proof actually shows that $\Delta$ is a \emph{strict} test category (see
  \cite[Section~1.6]{Maltsi}).
\end{rem}

\section{The tree category $\Omega$}

The tree category $\Omega$ was introduced by the third author and Weiss in
\cite{MoerdWeissDendSets}. The purpose of this section is to recall some of
the main definitions related to $\Omega$.

\begin{paragr}
By an \ndef{operad}, we will always mean a symmetric coloured operad. We
will denote by $\Oper$ the category of operads. If $P$ is an operad and
$c_1, \dots, c_n, d$ are colours, we will denote by $P(c_1, \dots, c_n; d)$
the set of operations in $P$ from $c_1, \dots, c_n$ to $d$.
\end{paragr}

\begin{paragr}
Similarly, by a \ndef{tree}, we will always mean a finite non-planar rooted
tree. Here is an example of such a tree:
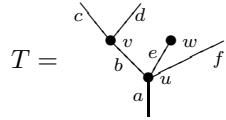
\[
T =
\xy
(0, 0)*{
\xy
(0, 0)*{}="a"; 
(0, 5)*=0{\bullet\annot{u}}="b";
(-5, 10)*=0{\bullet\annot{v}}="c1";
(3, 10)*=0{\bullet\annot{w}}="c2";
(10, 10)*{}="c3";
(-9, 15)*{}="e1";
(-1, 15)*{}="e2";
\ar@{-}"a";"b"^{a\!}
\ar@{-}"b";"c1"^(.75){b\!\!\!}
\ar@{-}"b";"c2"^(.45){e\!\!}
\ar@{-}"b";"c3"_(.95){\!\!\!f}
\ar@{-}"c1";"e1"^(.90){c\!\!}
\ar@{-}"c1";"e2"_(1){\!\!\!\!d}
\endxy}
\endxy
\]

Every tree $T$ generates a coloured operad $\Omega(T)$ in the following way.
Choose a planar structure on $T$ and consider the non-symmetric coloured
collection $\Omega_0(T)$ whose colours are the edges of $T$ and whose
operations are given by the vertices of~$T$ (the planar structure fixes the
source of such an operation). The operad $\Omega(T)$ is then the free
coloured operad on $\Omega_0(T)$. It does not depend on the choice of the
planar structure.

We will denote by $\eta$ the tree with one edge and no vertices, and, for $n
\ge 0$, by~$C_n$ the $n$-corolla, that is, the tree with one vertex and
$n$ leaves.
\end{paragr}

\begin{paragr}
The category $\Omega$ is defined in the following way. Its objects are
trees, and if $S$ and $T$ are two trees, a morphism $S \to T$ in $\Omega$ is
given by a morphism of operads $\Omega(S) \to \Omega(T)$. A presheaf on
$\Omega$ is called a \ndef{dendroidal set}. We will consider the Yoneda
embedding $\Omega \hookto \pref{\Omega}$ as an inclusion, thus identifying
each object of $\Omega$ with its associated dendroidal set.
\end{paragr}

\begin{paragr}
There is a fully faithful functor $i : \Delta \to \Omega$ defined by
\[
\Delta_n \mapsto
L_n =
\xy
(0, 0)*{
\xy
(0, 0)*{}="a"; 
(0, 5)*=0{\bullet}="b";
(0, 10)*=0{\bullet}="c"; 
(0, 15)*=0{\bullet}="d";
(0, 20)*=0{\bullet}="e";
(0, 25)*{}="f";
\ar@{-}"a";"b"_{n}
\ar@{-}"b";"c"_{n-1}
\ar@{.}"c";"d"
\ar@{-}"d";"e"_{1}
\ar@{-}"e";"f"_{0}
\endxy}
\endxy
\]
We will consider $i$ as an inclusion and we will thus identify $\Delta_n$
with $L_n$. In particular, $\Delta_0$ will be identified with $\eta$. The
image of $i$ being a sieve, the left Kan extension $i_! : \sSet \to \dSet$
sends a simplicial set $X$ to the dendroidal set obtained by extending $X$
by $\varnothing$ at trees not in the image of $i$. This functor is fully
faithful.
\end{paragr}

\begin{paragr}
Every map of $\Omega$ factors as a \ndef{degeneracy} followed by an
isomorphism followed by a \ndef{face map}. The face maps are generated by
elementary faces. An \ndef{elementary face} is either an inner face or
an outer face defined as follows. Let $T$ be a tree. For $e$ an inner edge
of~$T$ (that is an edge between two vertices), define $T/e$ to be the
tree obtained by contracting $e$. There is a map $\partial_e : T/e \to T$ in
$\Omega$ corresponding to the composition in $\Omega(T)$ of the two
operations associated to the end-vertices of~$e$. Such a map is called an
\ndef{inner face}. If $T$ has at least two vertices and $v$ is a vertex with
exactly one adjacent inner edge, define $T/v$ to be the tree obtained by
chopping off~$v$. There is a map $\partial_v : T/v \to T$ corresponding to
the obvious inclusion of operads. If $T$ has exactly one vertex, that is, if
$T$ is a corolla, there is one map $\eta \to T$ for each edge of $T$. These
two kinds of maps are called \ndef{outer faces}.

Similarly, degeneracies are generated by elementary degeneracies
defined as follows. If $T$ is a tree, for each edge $e$ there is an
\ndef{elementary degeneracy} $\sigma_e : S \to T$, where $S$ is obtained from $T$ by
inserting a vertex in the middle of $e$ and $\sigma_e$ corresponds to the
identity of $e$ in the operad $\Omega(T)$.

For more on faces and degeneracies, see \cite[Section
3]{MoerdWeissDendSets}.
\end{paragr}

\begin{paragr}
For a tree $T$ and an inner edge $e$ of $T$, we denote by
$\Lambda^e_T$ the maximal subobject of $T$ in $\dSet$ not containing the face
$\partial_e : T/e \to T$. The inclusions of the form $\Lambda^e_T \hookto T$
are called \ndef{inner horn inclusions}. A dendroidal set is an
\ndef{\oo-operad} if it has the extension property with respect to every
inner horn inclusion.
\end{paragr}

\begin{paragr}
A monomorphism of dendroidal sets $X \hookto Y$ is said to be \ndef{normal}
if, for any tree $T$, the action of the group of automorphisms of $T$ in
$\Omega$ on $Y(T)\backslash X(T)$ is free. The class of normal monomorphisms
can be characterised as the saturation (i.e., the closure under pushout, transfinite
composition and retracts) of the set~$\{\partial T \hookto T \mid T \in
\Omega\}$, where $\partial T$ denotes the maximal proper subdendroidal set
of $T$.

A dendroidal set $X$ is said to be \ndef{normal} if the monomorphism
$\varnothing \to X$ is normal. For instance, trees are normal dendroidal
sets. We will need the following two facts about normal dendroidal sets: if
$f : X \to Y$ is a map of dendroidal sets with $Y$ normal, then $X$ is normal;
if moreover $f$ is a monomorphism, then $f$ is a normal monomorphism.
\end{paragr}

We can now formulate one of the main results of \cite{CisMoerdDend}:

\begin{thm}
\tolerance=1000
  There exists a combinatorial model category structure on $\dSet$ whose
  cofibrations are the normal monomorphisms and whose fibrant objects are
  the \oo-operads.
\end{thm}

This model category structure will be called the \ndef{operadic model
category structure} and its weak equivalences the \ndef{operadic weak
equivalences}. These operadic weak equivalences can be characterised in
terms of Segal core inclusions that we now define.

\begin{paragr}
Let $T$ be a tree. For each vertex $v$ with $n$ input edges, there is a map
$C_n \to T$ corresponding to the operation associated to $v$ in $\Omega(T)$.
The \ndef{Segal core} $\SC{T}$ of $T$ is the smallest subdendroidal set of
$T$ containing the images of these maps. Inclusions of the form $\SC{T}
\hookto T$ are called \ndef{Segal core inclusions}.
\end{paragr}

The following characterisation of operadic weak equivalences follows
from \cite[Corollary 6.11]{CisMoerdDend} and \cite[Proposition
2.6]{CisMoerdDendSeg}:

\begin{thm}\label{thm:Segal}
  The class of operadic weak equivalences is the smallest class $\W$
  satisfying the following properties:
  \begin{enumerate}
    \item $\W$ satisfies the 2-out-of-3 property;
    \item $\W$ contains the class of maps having the right lifting property
    with respect to normal monomorphisms;
    \item the class of normal monomorphisms which are in $\W$ is closed
    under pushout, transfinite composition and retracts.
    \item $\W$ contains the set of Segal core inclusions.
  \end{enumerate}
\end{thm}

Recall finally that dendroidal sets are endowed with a tensor product.

\begin{paragr}
  We will denote by $\tensBV$ the Boardman--Vogt tensor product of operads
  (see \cite[Definition 2.14]{BoardVogt}).
  If $X$ and $Y$ are two dendroidal sets, their \ndef{tensor product} is defined by
  the formula
  \[ X \otimes Y = \limind_{S \to X,\, T \to Y} N_d(\Omega(S) \tensBV
  \Omega(T)),
  \]
  where $S$ and $T$ vary among trees and $N_d$ denotes the dendroidal nerve
  functor (see \cite[Example 4.2]{MoerdWeissDendSets}). This tensor product
  is symmetric but only associative up to weak equivalence (see
  \cite[Section 6.3]{HeutsHinichMoerd}). It admits $\eta$ as a unit.
  Moreover, it preserves colimits in each variable.
\end{paragr}

\section{The category $\Omega$ is aspherical}
\label{sec:Omega_asph}

The goal of this section is to show that the category $\Omega$ is
aspherical.

\begin{paragr}
The \ndef{décalage $D(T)$} of a tree $T$ is the tree 
\[ D(T) = T \amalg_{\eta} C_1, \]
where $\eta \to C_1$ is the map corresponding to the unique leaf of
$C_1$ and $\eta \to T$ is the root map of $T$, that is, the map corresponding
to the root of $T$. For each tree~$T$, we have a diagram
\[ T \xrightarrow{u^{}_T} D(T) \xleftarrow{a^{}_T} \eta, \]
where $u^{}_T : T \to D(T) = T \amalg_\eta C_1$ is the canonical map and
$a^{}_T : \eta \to D(T)$ is the root map of $D(T)$.
In other words, $D(T)$ is obtained from $T$ by adding a new unary vertex $v_T$ at the
root, with a new root edge $a_T$ coming out of it:
\[
T =
\xy
(0, 0)*{
\xy
(0, 0)*{}="a"; 
(0, 5)*=0{\bullet}="b";
(-5, 10)*{}="c1";
(5, 10)*{}="c2";
(0, 8)*{\cdots}
\ar@{-}"a";"b"^{r\!}
\ar@{-}"b";"c1"
\ar@{-}"b";"c2"
\endxy}
\endxy
\quad
\longmapsto
\quad
D(T) =
\xy
(0, 0)*{
\xy
(0, -5)*{}="r"; 
(0, 0)*=0{\bullet\annot{v^{}_T}}="a"; 
(0, 5)*=0{\bullet}="b";
(-5, 10)*{}="c1";
(5, 10)*{}="c2";
(0, 8)*{\cdots}
\ar@{-}"r";"a"^{a^{}_T\!}
\ar@{-}"a";"b"^{r\!}
\ar@{-}"b";"c1"
\ar@{-}"b";"c2"
\endxy}
\endxy
\]
The map $u^{}_T : T \to D(T)$ defined above is the outer face map
$\partial_{v^{}_T}$.
\end{paragr}

\begin{rem}
It is tempting to think at this point that we have a zigzag of natural
maps
\[ T \xrightarrow{u_T} D(T) \xleftarrow{a_T} \eta \]
and hence that $\Omega$ is aspherical. Notice though that we have not defined
the action of $D$ on the morphisms of $\Omega$. It turns out that there is
no way to make $D$ into a functor $D : \Omega \to \Omega$ for which the
maps
\[ T \xrightarrow{u^{}_T} D(T) \xleftarrow{a^{}_T} \eta \]
are natural. Indeed, consider the face map $\partial_v : S \to T$:
\[
S =
\xy
(0, 0)*{
\xy
(0, 0)*{}="a"; 
(0, 5)*=0{\bullet\annot{w}}="b";
(-5, 10)*{}="c1";
(0, 10)*{}="c2";
(5, 10)*{}="c3";
\ar@{-}"a";"b"^{c\!}
\ar@{-}"b";"c1"^(1.2){d\!\!\!}
\ar@{-}"b";"c2"^(.75){e\!}
\ar@{-}"b";"c3"_(1.2){\!\!\!\!f}
\endxy}
\endxy
\quad\longrightarrow\quad
T =
\xy
(0, 0)*{
\xy
(0, 0)*{}="a"; 
(0, 5)*=0{\bullet\annot{v}}="b";
(-5, 10)*{}="c1";
(5, 10)*=0{\bullet\annot{w}}="c2";
(0, 15)*{}="d1";
(5, 15)*{}="d2";
(10, 15)*{}="d3";
\ar@{-}"a";"b"^{a\!}
\ar@{-}"b";"c1"^(.88){b\!\!\!}
\ar@{-}"b";"c2"_(.75){\!\!\!c}
\ar@{-}"c2";"d1"^(1.2){d\!\!\!}
\ar@{-}"c2";"d2"^(.75){e\!}
\ar@{-}"c2";"d3"_(1.2){\!\!\!\!f}
\endxy}
\endxy
\]
Clearly, this map cannot be extended to a root-preserving map
\[
D(S) =
\xy
(0, 0)*{
\xy
(0, -5)*{}="r"; 
(0, 0)*=0{\bullet\annot{v^{}_S}}="a"; 
(0, 5)*=0{\bullet}="b";
(-5, 10)*{}="c1";
(0, 10)*{}="c2";
(5, 10)*{}="c3";
\ar@{-}"r";"a"^{a^{}_S\!}
\ar@{-}"a";"b"^{c\!}
\ar@{-}"b";"c1"^(1.2){d\!\!\!}
\ar@{-}"b";"c2"^(.75){e\!}
\ar@{-}"b";"c3"_(1.2){\!\!\!\!f}
\endxy}
\endxy
\quad\longrightarrow\quad
D(T) =
\xy
(0, 0)*{
\xy
(0, -5)*{}="r"; 
(0, 0)*=0{\bullet\annot{v^{}_T}}="a"; 
(0, 5)*=0{\bullet\annot{v}}="b";
(-5, 10)*{}="c1";
(5, 10)*=0{\bullet\annot{w}}="c2";
(0, 15)*{}="d1";
(5, 15)*{}="d2";
(10, 15)*{}="d3";
\ar@{-}"r";"a"^{a^{}_T\!}
\ar@{-}"a";"b"^{a\!}
\ar@{-}"b";"c1"^(.88){b\!\!\!}
\ar@{-}"b";"c2"_(.75){\!\!\!c}
\ar@{-}"c2";"d1"^(1.2){d\!\!\!}
\ar@{-}"c2";"d2"^(.75){e\!}
\ar@{-}"c2";"d3"_(1.2){\!\!\!\!f}
\endxy}
\endxy
\]
for $v^{}_S$ would have to be sent to a unary operation in $\Omega(D(T))$
from $c$ to $a^{}_T$ and there is no such operation. Note that if there
were a (nullary) vertex above $b$ in~$T$, then there would exist such an
operation. This leads to the following definition.
\end{rem}

\begin{defi}
A tree $T$ is said to be \ndef{closed} if it has no leaves, that is, if there is
a vertex above every edge of $T$. The full subcategory of $\Omega$
consisting of closed trees is denoted $\OmegaCl$.
\end{defi}

\begin{paragr}
The \ndef{closure} of a tree $T$ is the tree $\cl(T) = \overline{T}$
obtained from $T$ by adjoining a (nullary) vertex $v_l$ on top of each leaf
$l$ of $T$. We will denote by $\eta_T : T \hookrightarrow \overline{T}$ the
obvious inclusion.

If $f : S \to T$ is a map in $\Omega$, we define $\cl(f) = \overline{f} :
\overline{S} \to \overline{T}$ to be the unique map $\overline{S} \to
\overline{T}$ extending $S \to T$, in the sense that the diagram
\[
\xymatrix@C=2.5pc{
\overline{S} \ar[r]^{\overline{f}} & \overline{T} \\
S \ar[u]^{\eta^{}_S} \ar[r]_f & T \ar[u]_{\eta^{}_T}
}
\]
commutes. Since this property determines the action of $\overline{f}$ on the edges of
$\overline{S}$, there is at most one such map. Its existence follows from the
fact that $\Omega(\overline{T})$ has a (unique) nullary operation for
each of its colours.

One checks that this defines a functor $\cl : \Omega \to
\OmegaCl$ from trees to closed trees.
\end{paragr}

\begin{ex}
Consider the following external face map $\partial_w$:
\[
R =
\xy
(0, 0)*{
\xy
(0, 0)*{}="a"; 
(0, 5)*=0{\bullet\annot{v}}="b";
(-5, 10)*{}="c1";
(5, 10)*{}="c2";
\ar@{-}"a";"b"^{a\!}
\ar@{-}"b";"c1"^(.88){b\!\!\!}
\ar@{-}"b";"c2"_(.75){\!\!\!c}
\endxy}
\endxy
\quad\longrightarrow\quad
T =
\xy
(0, 0)*{
\xy
(0, 0)*{}="a"; 
(0, 5)*=0{\bullet\annot{v}}="b";
(-5, 10)*{}="c1";
(5, 10)*=0{\bullet\annot{w}}="c2";
(0, 15)*{}="d1";
(5, 15)*{}="d2";
(10, 15)*{}="d3";
\ar@{-}"a";"b"^{a\!}
\ar@{-}"b";"c1"^(.88){b\!\!\!}
\ar@{-}"b";"c2"_(.75){\!\!\!c}
\ar@{-}"c2";"d1"^(1.2){d\!\!\!}
\ar@{-}"c2";"d2"^(.75){e\!}
\ar@{-}"c2";"d3"_(1.2){\!\!\!\!f}
\endxy}
\endxy
\]
The closure $\cl(\partial_w)$ of $\partial_w$
\[
\overline{R} =
\xy
(0, 0)*{
\xy
(0, 0)*{}="a"; 
(0, 5)*=0{\bullet\annot{v}}="b";
(-5, 10)*=0{\bullet}="c1";
(5, 10)*=0{\bullet}="c2";
\ar@{-}"a";"b"^{a\!}
\ar@{-}"b";"c1"^(.88){b\!\!\!}
\ar@{-}"b";"c2"_(.75){\!\!\!c}
\endxy}
\endxy
\quad\longrightarrow\quad
\overline{T} =
\xy
(0, 0)*{
\xy
(0, 0)*{}="a"; 
(0, 5)*=0{\bullet\annot{v}}="b";
(-5, 10)*=0{\bullet}="c1";
(5, 10)*=0{\bullet\annot{w}}="c2";
(0, 15)*=0{\bullet}="d1";
(5, 15)*=0{\bullet}="d2";
(10, 15)*=0{\bullet}="d3";
\ar@{-}"a";"b"^{a\!}
\ar@{-}"b";"c1"^(.88){b\!\!\!}
\ar@{-}"b";"c2"_(.75){\!\!\!c}
\ar@{-}"c2";"d1"^(1.2){d\!\!}
\ar@{-}"c2";"d2"^(.75){e\!}
\ar@{-}"c2";"d3"_(1.2){\!\!\!f}
\endxy}
\endxy
\]
 is the composition of the three inner face maps $\partial_d$, $\partial_e$
 and $\partial_f$.
\end{ex}

\begin{prop}\label{prop:OmegaCl_adj}
The inclusion $i :  \OmegaCl \hookrightarrow \Omega$ admits the functor $\cl
: \Omega \to \OmegaCl$ as a left adjoint.
\end{prop}

\begin{proof}
By the previous paragraph, we have a natural transformation $\eta :
\id{\Omega} \to i\cl$. Denote by $\eps : \cl i = \id{\OmegaCl} \to
\id{\OmegaCl}$ the identity natural transformation. We claim that $\eta$ and
$\eps$ are the unit and counit of the announced adjunction. Using the fact
that $\eps$ is the identity and that $\eta$ is the identity on $\OmegaCl$, the
triangular identities reduce to the equality $\cl\eta = \id{\cl}$; that is,
to the fact that if $T$ is a tree, we have $\cl(T \to \cl(T)) =
\id{\cl(T)}$. This is readily checked.
\end{proof}

\begin{paragr}
We now define a functor $D : \OmegaCl \to \OmegaCl$ extending the
assignment $T \mapsto D(T)$ restricted to closed trees. If $f : S
\to T$ is map in $\OmegaCl$, we define $D(f) : D(S) \to D(T)$ to be the
unique root-preserving map $D(S) \to D(T)$ extending $S \to T$, in the sense
that the diagram
\[
\xymatrix@C=2.5pc{
D(S) \ar[r]^{D(f)} & D(T) \\
S \ar[u]^{u^{}_S} \ar[r]_f & T \ar[u]_{u^{}_T}
}
\]
commutes. Since this property determines the action of $D(f)$ on the edges of
$D(S)$, there is at most one such map. Its existence follows from the
fact that for every edge~$e$ of $D(T)$ (or more generally of any closed
tree), there is a (unique) unary operation from $e$ to the root of $D(T)$ in
$\Omega(D(T))$.

One checks that this indeed defines a functor $D : \OmegaCl \to \OmegaCl$.
\end{paragr}

\begin{prop}\label{prop:dec}
The maps
\[ T \xrightarrow{u^{}_T} D(T) \xleftarrow{\overline{a^{}_T}} \overline{\eta} \]
are natural in $T$ in $\OmegaCl$.
\end{prop}

\begin{proof}
The naturality of $u^{}_T$ is true by definition. The one of
$\overline{a^{}_T}$ boils down to the fact that for any map $f$ in
$\OmegaCl$, the map $D(f)$ is root-preserving.
\end{proof}

\begin{rem}
The diagram
\[ \id{\OmegaCl} \xrightarrow{u} D \xleftarrow{\overline{a}} \overline{\eta} \]
is a ``split décalage'' in the sense of \cite[paragraph
3.1]{CisMaltsiTheta}. This implies that $\OmegaCl$ is a (strict) test
category (see \cite[Corollary 3.7]{CisMaltsiTheta}).
\end{rem}

\begin{thm}\label{thm:Omega_asph}
The category $\Omega$ is aspherical.
\end{thm}

\begin{proof}
By Proposition~\ref{prop:OmegaCl_adj}, $\Omega$ is aspherical if and only if
$\OmegaCl$ is. But the asphericity of $\OmegaCl$ follows
from Proposition~\ref{prop:dec}.
\end{proof}

\section{The category $\Omega$ is a test category}

In this section, we will prove that $\Omega$ is a test category. Our proof
is based on the fact that for any tree $T$, the dendroidal set $\Delta_1
\otimes T$ is aspherical. More generally and parallel to the reasoning in
Section~\ref{sec:Delta}, we observe that for any two trees $S$ and~$T$,
their tensor product $S \otimes T$ is aspherical. To prove this,
we will need the shuffle formula introduced in \cite[Section
9]{MoerdWeissInnKan}:

\begin{prop}
Let $S$ and $T$ be two trees. The dendroidal set $S \otimes T$ can be
written as a finite union of subdendroidal sets
\[
  S \otimes T = \bigcup_{\sigma \in \Sh{S}{T}} F_\sigma
\]
satisfying the following properties:
\begin{enumerate}
  \item the $F_\sigma$'s are representable;
  \item the $F_\sigma$'s have the same root and leaves, seen as elements
    of~$(S \otimes T)(\eta)$;
  \item the $F_\sigma$'s are full subdendroidal sets of $S \otimes T$, where $X \subset
  Y$ is said to be \ndef{full} if an element of $Y(U)$, for $U$ a tree,
  belongs to $X(U)$ if and only if all its faces in $Y(\eta)$ belong to
  $X(\eta)$.
\end{enumerate}
\end{prop}

\begin{rem}
The indexing set $\Sh{S}{T}$ in the above formula is the set of $(S,
T)$\nbd-shuffles introduced in \cite{MoerdWeissInnKan} under the name of
``percolation schemes for $S$ and $T$''.
\end{rem}

\begin{prop}\label{prop:prod_trees}
If $S$ and $T$ are two trees, then the dendroidal set $S \otimes T$ is
aspherical.
\end{prop}

\begin{proof}
By Lemma~\ref{lemma:MV} applied to the formula
\[
  S \otimes T = \bigcup_{\sigma \in \Sh{S}{T}} F_\sigma,
\]
it suffices to prove that for any non-empty $J \subseteq \Sh{S}{T}$, the
dendroidal set
\[
F_J = \bigcap_{j \in J} F_j \subset S \otimes T
\]
is aspherical. We will actually show that it is representable. As all the
$F_j$'s are full subdendroidal sets of $S \otimes T$, the intersection $F_J$
is the (unique) full subdendroidal set such that $F_J(\eta) = \cap_{j \in J}
F_j(\eta)$. In particular, for any $j \in J$, $F_J$ is a full subdendroidal
set of $F_j$ containing the root and the leaves of $F_j$. This implies that
$F_J$ is an iterated inner face of $F_j$ and is hence representable.
\end{proof}

We will now see that $\Delta_1 \otimes X$ can be thought of as a cylinder
object, at least when $X$ is normal.

\begin{paragr}
  If $X$ is a dendroidal set, we have canonical maps
  \[ X \amalg X \xrightarrow{(\partial^0, \partial^1)} \Delta_1 \otimes X
  \xrightarrow{\sigma} X, \]
  factorising the codiagonal, induced by the diagram
  \[ \Delta_0 \amalg \Delta_0 = \partial{\Delta_1} \hookrightarrow \Delta_1
  \to \Delta_0, \]
  tensored by $X$.
\end{paragr}

\begin{prop}\label{prop:cylinder}
  If $X$ is a normal dendroidal set, then
  \begin{enumerate}
    \item the map $(\partial^0, \partial^1) : X \amalg X \to \Delta_1
    \otimes X$ is a normal monomorphism;
    \item the map $\sigma : \Delta_1 \otimes X \to X$ is a weak equivalence.
  \end{enumerate}
\end{prop}

\begin{proof}
  The first assertion follows since the tensor product of a
  monomorphism of simplicial sets and a normal dendroidal set is a normal
  monomorphism (see \hbox{\cite[Section 3.4]{HeutsHinichMoerd}}).

  The second assertion is a special case of Proposition
  ~\ref{prop:prod_trees} if $X$ is representable. Since the functor $X
  \mapsto \Delta_1 \otimes X$ preserves colimits and monomorphisms between
  normal objects (see loc.~cit.), the general case follows by a standard
  induction on normal objects (see for instance
  \cite[Proposition~8.2.8]{Cisinski}), using the fact that the functor
  $\lambda_!$ of \ref{paragr:def_lambda} preserves colimits and
  monomorphisms, and detects weak equivalences.
\end{proof}

\begin{thm}
  The category $\Omega$ is a test category.
\end{thm}

\begin{proof}
 Since $\Omega$ is aspherical (Theorem~\ref{thm:Omega_asph}), by
 Theorem~\ref{thm:loc_test} it suffices to show that the Lawvere interval
 $L_\Omega$ is locally aspherical; that is, that for every
 tree~$T$, the projection $p : L_\Omega \times T \to T$ is a weak
 equivalence. The map $p$ satisfies the two following properties: first, it
 has the right lifting property with respect to monomorphisms (as $L_\Omega$
 is injective); second, its source is normal (as its target is). The
 following standard argument shows that any map $p : X \to Y$ satisfying
 these two conditions is a weak equivalence. Note first that such a map
 admits a section~$s$. Consider now the commutative square
  \[
  \xymatrix{
    X \amalg X
    \ar[rr]^-{(\id{}, sp)}
    \ar[d]_{(\partial^0, \partial^1)} &&
    X
    \ar[d]^p \\
    \Delta_1 \otimes X \ar[r]_\sigma &
    X \ar[r]_p
    \ar[r] &
    Y \pbox{.}
  }
  \]
  Since by the previous proposition $(\partial^0, \partial^1)$ is a
  monomorphism, this square admits a lifting $h : \Delta_1 \otimes X \to X$
  defining a homotopy from $1$ to $sp$ in
  an obvious sense. As $\sigma : \Delta_1 \otimes X \to X$ is a weak
  equivalence (again by the previous proposition), this implies that $sp$ is
  a weak equivalence and so $p$ has a left as well as a right inverse in the
  homotopy category, hence is a weak equivalence, thereby proving
  the result.
\end{proof}

\begin{rem}
  The category $\Omega$ is not a \ndef{strict} test category (see
  \cite[Section~1.6]{Maltsi}) as the category $\Omega/(\eta \times C_2)
  \simeq \Delta/C_2$, where $C_2$ is the $2$-corolla, is not aspherical (it
  is not even connected).
\end{rem}

\begin{rem}
  The planar variation $\OmegaP$ of $\Omega$, defined by replacing symmetric
  operads by non-symmetric operads, is also a test category. Indeed,
  $\OmegaP$ is equivalent to the slice category of $\Omega$ over the
  presheaf of planar structures and by~\cite[Remark~1.5.4]{Maltsi}, it
  suffices to observe that the same argument as in
  Section~\ref{sec:Omega_asph} shows that $\OmegaP$ is aspherical.
  Similarly, the category $\OmegaPCl$ of closed planar trees is a strict
  test category.
\end{rem}

\begin{coro} \label{coro:test_model}
  There exists a proper combinatorial model category structure on~$\dSet$
  whose cofibrations are the monomorphisms and whose weak equivalences are
  the weak equivalences defined in \ref{paragr:def_weq_psh}. Moreover, the
  functor $\lambda_! : \dSet \to \pref{\Delta}$ of~\ref{paragr:def_lambda}
  is a left Quillen equivalence.
\end{coro}

\begin{proof}
  This follows from the previous theorem using
  Theorems~\ref{thm:model_cat_local_test} and~\ref{thm:test_model}.
\end{proof}

The model category structure of the previous corollary will be called the
\ndef{test model category structure}. From now on, its weak equivalences
will be called \ndef{test weak equivalences} to distinguish them from
operadic weak equivalences.

\section{Comparison with the operadic model category structure}

In this last section, we compare the test model category structure on
$\dSet$ with the operadic model category structure.

\begin{prop}\label{prop:norm_fib}
  Every map of $\pref{\Omega}$ having the right lifting property with respect
  to normal monomorphisms is a test weak equivalence.
\end{prop}

\begin{proof}
  Let $p : X \to Y$ be such a map. By Quillen's Theorem A, to prove that $p$
  is a test weak equivalence, it suffices to check that for any tree $T$ and
  any map $T \to Y$, the map $q = p \times_Y T : X \times_Y T \to T$ is
  a test weak equivalence. But this map $q$ has the same lifting property as
  $p$ and, if
  \[
  \xymatrix{
    A \ar[r] \ar[d]_i \ar[r] & X \times_Y T \ar[d]^q \\
    B \ar[r] & T
  }
  \]
  is a commutative square where $i$ is a monomorphism, then $i$ is
  automatically normal as $T$ and hence $B$ are normal. In particular, $q$ is
  a trivial fibration in the test model category structure and hence a test
  weak equivalence, thereby proving the result.
\end{proof}

\begin{thm}\label{thm:norm_mod_struct}
  There exists a proper model category structure on~$\pref{\Omega}$ whose
  cofibrations are the normal monomorphisms and whose weak equivalences are
  the test weak equivalences.
\end{thm}

This model category structure on $\dSet$ will be called the \ndef{normal
test model category structure}.

\begin{proof}
  The existence of a model category structure with the announced weak
  equivalences and cofibrations is a consequence of the following lemma
  applied to the test model category structure, the hypothesis of the lemma
  being satisfied by the previous proposition. As properness only depends on
  the weak equivalences, the properness of the resulting model category
  structure follows from the properness of the test model category
  structure.
\end{proof}

\begin{lemma}
  Let $\M$ be a model category. Suppose $\Cof'$ is a class of cofibrations,
  which is the saturation of a set (rather than a class) of morphisms
  allowing the small object argument, and has the property that any map
  having the right lifting property with respect to $\Cof'$ is a weak
  equivalence. Then there exists a model category structure on $\M$ with the
  same weak equivalences, and $\Cof'$ as class of cofibrations.
\end{lemma}

\begin{proof}
  This lemma is probably well-known to experts, but we include a proof for
  completeness. Let us denote by $\W$, $\Cof$ and $\Fib$ the classes of weak
  equivalences, cofibrations and fibrations of $\M$, and by $\Fib'$ the
  class of maps having the right lifting property with
  respect to $\Cof' \cap \W$. Let us show that $(\M, \W, \Cof', \Fib')$ is a
  model category. The only axioms that are not obviously true are the
  lifting axiom and the factorisation axiom.

  Let us start by the factorisation axiom. By the small object argument,
  every map~$f$ of $\M$ factors as $f = pi$, where $i$ is in $\Cof'$ and $p$
  has the right lifting property with respect to $\Cof'$. By hypothesis,
  such a $p$ is in $\W$. This shows that $p$ is in~$\Fib' \cap \W$.
  As for the second factorisation, if $f$ is map of $\M$, we can write
  $f = pi$, where $i$ is in $\Cof \cap \W$ and $p$ is in $\Fib$. Using the
  previous factorisation, we get $i = qj$, where $j$ is in $\Cof'$ and $q$
  is in $\Fib' \cap \W$. As $\Cof' \subset \Cof$, we have $\Fib \subset
  \Fib'$, showing that $pq$ is in $\Fib'$, so that $f = (pq)j$ is a
  factorisation of the desired kind.

  To conclude the proof, we observe that one half of the lifting axiom holds
  by definition, while the other half follows by the following standard
  retract argument. Consider a map $f$ in~$\Fib' \cap \W$. We can factor $f$ as
  $pi$ with $i$ in $\Cof'$ and $p$ having the right lifting property with
  respect to $\Cof'$ and hence being in $\W$. By the 2-out-of-3 property,
  $i$~is in $\W$, and hence in $\Cof' \cap \W$. This implies that $f$ has
  the right lifting property with respect to $i$ and hence, by the retract
  lemma, that $f$ is a retract of~$p$, and so that it has the right lifting
  property with respect to the class $\Cof'$.
\end{proof}


\begin{prop}
  Every operadic weak equivalence is a test weak equivalence.
\end{prop}

\begin{proof}
  Theorem~\ref{thm:norm_mod_struct} implies that the class of test weak
  equivalences satisfies the first three conditions of
  Theorem~\ref{thm:Segal}. This shows that the assertion is equivalent to
  the fact that Segal core inclusions are test weak equivalences. It thus
  suffices to prove that for any tree $T$, the Segal core of $T$ is
  aspherical. This follows from the fact that the Segal core of $T$
  can be constructed by iteratively gluing corollas along~$\eta$.
\end{proof}

\begin{thm}
  The normal test model category structure on $\pref{\Omega}$ is the left
  Bousfield localisation of the operadic model category structure by the set
  of maps between representable dendroidal sets.
\end{thm}

\begin{proof}
  Let $E$ be a normalisation of the terminal dendroidal set, that is, a
  normal dendroidal set such that the map $p$ to the terminal dendroidal set
  has the right lifting property with respect to normal monomorphisms.
  Consider the adjunction
  \[ p_! : \pref{\Omega}/E \rightleftarrows \pref{\Omega} : p^\ast, \]
  where $p_!$ is the forgetful functor and $p^\ast$ the functor sending $X$
  to $X \times E$. For any dendroidal set $X$, the projection $X \times E \to X$ has
  the right lifting property with respect to normal monomorphisms and is
  hence an operadic weak equivalence. This shows that the unit and the
  counit of the adjunction $(p_!, p^\ast)$ are objectwise operadic weak
  equivalences and we have a Quillen equivalence
  \[ p_! : \pref{\Omega}_\operadic/E \rightleftarrows
  \pref{\Omega}_\operadic : p^\ast, \]
  where $\pref{\Omega}_\operadic$ denotes the operadic model category
  (see also \cite[proof of Proposition~3.12]{CisMoerdDend}). Note that the
  fact that $E$ is normal implies that the cofibrations
  of~$\pref{\Omega}_\operadic/E$ are the monomorphisms. Consider the left
  Bousfield localisation of this Quillen equivalence by the set~$S$ of maps
  between representables of $\pref{\Omega}/E \simeq
  \text{\smash{$\pref{\Omega/E}$}}$. We get a Quillen equivalence
   \[
     p_! : L_S(\pref{\Omega}_\operadic/E) \rightleftarrows
     L_{S'}(\pref{\Omega}_\operadic) : p^\ast,
   \]
  where $S'$ denotes the set of maps between representable dendroidal sets,
  the unit and counit still being objectwise weak equivalences.
  As every object of $L_S(\pref{\Omega}_\operadic/E)$ is cofibrant,
  a map of dendroidal sets $f$ is a weak equivalence of
  $L_{S'}(\pref{\Omega}_\operadic)$ if and only if $p^*(f)$ is in the class
  $\W$ of weak equivalences of $L_S(\pref{\Omega}_\operadic/E)$. Similarly,
  as $X \times E \to X$ is a test weak equivalence for any dendroidal set
  $X$, such a map~$f$ is a test weak
  equivalence if and only if $p^\ast(f)$ is in the class $\W_\infty$ of test
  weak equivalences of~$\pref{\Omega}/E \simeq
  \text{\smash{$\pref{\Omega/E}$}}$, which is nothing but the class of test
  weak equivalences of $\pref{\Omega}$ above~$E$. To conclude the proof, it
  thus suffices to show the equality $\W = \W_\infty$.

  By the previous proposition, the identity functor of $\pref{\Omega}/E$ is
  a left Quillen functor from the operadic model category to the normal test
  model category (or more precisely between their slices). As $S$ belongs to
  $\W_\infty$, the universal property of localisations implies that $\W
  \subset \W_\infty$.

  To prove the converse, we will use the machinery of \cite{Cisinski}.
  Define an \ndef{$\Omega/E$-localiser} to be the class of weak equivalences
  of some combinatorial model category structure on $\pref{\Omega}/E$ whose
  cofibrations are the monomorphisms. (By \cite[Theorem 1.4.3]{Cisinski},
  this is equivalent to what is called an \emph{accessible} localiser in
  \cite[Section 1.4]{Cisinski}.) Since $E$ is normal,
  the category~$\Omega/E$ is a regular skeletal category in the sense
  of \cite[Definition~8.2.5]{Cisinski}. Moreover, as $E$ is aspherical,
  $\Omega/E$ is a test category (see \cite[Remark 1.5.4]{Maltsi}). In
  particular, by \cite[Proposition 6.4.26 and Proposition 8.2.9]{Cisinski},
  $\W_\infty$~is the smallest $\Omega/E$\nbd-localiser containing $S$. As
  $\W$ is a localiser containing $S$, we have $\W_\infty \subset \W$,
  thereby ending the proof.
\end{proof}

\begin{rem}
  The previous theorem can also be proved as follows. One identifies the
  operadic model category structure $\pref{\Omega}_\operadic$ on dendroidal sets
  with a localisation of the category of dendroidal spaces (that is,
  simplicial presheaves on $\Omega$) equipped with the Reedy model category
  structure. This localisation involves the inner horn inclusions
  $\Lambda^e_T \hookto T$ and the map $\{0\} \hookto J$, where $J$ denotes
  the nerve of the simply connected groupoid on two objects $0$ and $1$, see
  \cite[Sections 5 and~6]{CisMoerdDendSeg}. Thus the localisation of
  $\pref{\Omega}_\operadic$ by the maps between representables is equivalent
  to a further localisation of dendroidal spaces by the images $S \to T$ of
  these maps. One checks that localising the Reedy model category
  structure on dendroidal spaces by these maps $S \to T$ already makes the
  inner horn inclusions (or equivalently, the Segal core inclusions) weak
  equivalences, as well as the image of any Kan--Quillen weak equivalence of
  simplicial sets under the embedding of simplicial spaces into dendroidal
  spaces, so in particular $\{0\} \to J$. But the localisation thus obtained
  describes the homotopy theory of homotopically constant contravariant
  diagrams of spaces on~$\Omega$, hence is equivalent to that of spaces
  since $\Omega$ is aspherical.
\end{rem}

\bibliographystyle{plain}
\bibliography{biblio}

\end{document}